\documentclass{amsart}
\usepackage[hmargin=3.5cm,vmargin=3.5cm]{geometry}
\usepackage[utf8]{inputenc}

\usepackage{amsmath, 
	,amssymb
	,amsthm
	,stmaryrd
	, mathtools, mathrsfs
}
\usepackage[linktocpage=true, backref = page, pdfusetitle]{hyperref}

\usepackage{graphicx,svg,xcolor}
\usepackage{enumitem,blindtext}
\usepackage{todonotes, lipsum}
\usepackage[all,cmtip]{xy}

\makeatletter
\g@addto@macro\bfseries{\boldmath}
\makeatother

\newtheorem{theorem}{Theorem}[section]
\newtheorem{lemma}[theorem]{Lemma}
\newtheorem{korollar}[theorem]{Corollary}

\newtheorem{proposition}[theorem]{Proposition}
\theoremstyle{definition}
\newtheorem{example}[theorem]{Example}
\newtheorem{remark}[theorem]{Remark}
\newtheorem{definition}[theorem]{Definition}

\newtheorem*{theorem*}{Theorem}
\newtheorem*{corollary*}{Corollary}
\newtheorem*{fact}{Fact}

\newcommand{\G}{\ensuremath{\Gamma}}
\newcommand{\g}{\ensuremath{\gamma}}

\newcommand{\C}{\ensuremath{\mathbb{C}}}

\newcommand{\eps}{\ensuremath{\varepsilon}}

\newcommand{\N}{\ensuremath{\mathbb{N}}}

\newcommand{\cf}{\ensuremath{\mathbf{c}}}
\newcommand{\ov}{\ensuremath{\overline}}
\newcommand{\mf}{\ensuremath{\mathfrak}}
\newcommand{\mc}{\ensuremath{\mathcal}}

\newcommand{\R}{\ensuremath{\mathbb{R}}}
\newcommand{\inv}{\ensuremath{^{-1}}}

	\renewcommand{\Re}{\ensuremath{\operatorname{Re}}}

	\newcommand{\on}{\ensuremath{\operatorname}}
	
	\DeclareMathOperator{\ad}{ad}

	\DeclareMathOperator{\HC}{HC}

	\DeclareMathOperator{\dom}{dom}

	\makeatletter
	\newcommand*\bigcdot{\mathpalette\bigcdot@{.5}}
	\newcommand*\bigcdot@[2]{\mathbin{\vcenter{\hbox{\scalebox{#2}{$\m@th#1\bullet$}}}}}
	\makeatother

	\author{Tobias Weich and Lasse L. Wolf}
	\title[Temperedness of locally symmetric spaces: The product case]{Temperedness of locally symmetric spaces:\\ The product case}
	\email{weich@math.upb.de, llwolf@math.upb.de}
	
\begin{document}
	\begin{abstract}
		Let $X=X_1\times X_2$ be a product of two rank one symmetric spaces of non-compact type and $\G$ a torsion-free discrete subgroup in $G_1\times G_2$.
		We show that the spectrum of $\G\backslash X$ is related to the asymptotic growth of $\G$ in the two direction defined by the two factors.
		We obtain that $L^2(\G\backslash G)$ is tempered for large class of $\G$.
	\end{abstract}

	\maketitle
	\section{Introduction}
	If one considers a geometrically finite hyperbolic surface $M = \Gamma \backslash \mathbb H$ it is a very classical theorem that the smallest eigenvalue of the Laplace-Beltrami operator $\Delta$ is related to the growth rate of $\Gamma$. 
	More precisely,
	\[ 
		\inf \sigma(\Delta) =\begin{cases}
			1/4&\colon\delta_\G<1/2\\
			1/4-(\delta_\G-1/2)^ 2&\colon \delta_\G\geq1/2,
		\end{cases} 
	\]
where $\delta_\Gamma$ is the critical exponent of the discrete subgroup $\Gamma\subseteq SL_2(\R)$
\[
	\delta_\G \coloneqq \inf\left\{s\in \R \colon \sum_{\g\in\G} e^{-s d(\gamma x_0, x_0)}<\infty \right\}, \quad x_0\in \mathbb{H}.
\]
This theorem is due to Elstrodt \cite{MR360472,MR360473,MR360474} and Patterson \cite{pattersonlimitset}.

A decade later it has been extended to real hyperbolic manifolds of arbitrary dimension by Sullivan \cite{Sul87} and then to general locally symmetric spaces of rank one by Corlette \cite{Cor90}.

We are interested in analog statements for higher rank locally symmetric spaces.
To state the theorems let us shortly introduce the setting (see Section~\ref{sec:setting}).
Let $X$ be a symmetric space of non-compact type, 
i.e. $X=G/K$ where $G$ is a real connected semisimple non-compact Lie group with finite center and $K$ is a maximal compact subgroup.
$G$ admits a Cartan decomposition $G=K \exp(\overline{\mf a_+}) K$.
Hence for every $g\in G$ there is $\mu_+(g)\in \overline{\mf a_+}$ such that $g\in K\exp(\mu_+(g))K$.
$\mu(g)$ can be thought of a higher dimensional distance $d(gK,eK)$.

In this setting the bottom of the spectrum of the Laplace-Beltrami operator $\Delta$ can be estimated using $\delta_\Gamma$ as well \cite{Web08, Leu04}.
Note that in the definition of $\delta_\Gamma$ the term $d(\gamma K, eK)$ is $\|\mu_+(\gamma)\|$.
Hence, one only considers the norm of $\mu_+(\gamma)$ but there are different ways to measure the growth rate of $\gamma$ or $\mu_+(\gamma)$.
This is exploited by Anker and Zhang \cite{ankerzhang} to determine $\inf\sigma(\Delta)$ to an exact value.

However, the spectral theory of $\Gamma\backslash G/K$ is more involved than in the rank one case and is not completely determined by $\Delta$:
There is a whole algebra of natural differential operators on $\Gamma\backslash G/K$ that come from the algebra of $G$-invariant differential operators $\mathbb{D}(G/K)$ on $G/K$.
In the easiest higher rank example $G/K=(G_1\times G_2)/(K_1\times K_2)=(G_1/K_1)\times (G_2/K_2)$ of two rank one symmetric spaces this algebra is generated by the two Laplacians acting on the respective factors.
In this case we could just consider the Laplace operators on the two factors $G_1/K_1$ and $G_2/K_2$ which generate $\mathbb{D}((G_1\times G_2)/(K_1\times K_2))$.
However, in general there are no canonical generators for $\mathbb{D}(G/K)$.
This is the reason why in the higher rank setting it is more natural to work with the whole algebra instead of a generating set.

The importance of this algebra can be seen by considering the representation $L^2(\Gamma\backslash G)$ where $G$ acts by right translation.
In the rank one case (where $\mathbb{D}(G/K) =\C[\Delta]$) $L^2(\Gamma\backslash G)$ is tempered (see Definition~\ref{def:tempered}) if $\sigma(\Delta) \subseteq [\|\rho\|^2,\infty[$.
In the higher rank case this is not true anymore but an analogous statement can be formulated in terms of $\mathbb{D}(G/K)$ (see Proposition~\ref{prop:temperednessfromspectrum}).
This requires to define a joint spectrum $\widetilde \sigma(\Gamma\backslash G/K)$ for $\mathbb{D}(G/K)$ on $L^2(\Gamma\backslash G/K)$.
There are different ways to define this spectrum:
On the one hand we can use the representation theoretical decomposition of $L^2(\Gamma\backslash G)$ and consider the support of the corresponding measure (see Section~\ref{sec:planchrelspectrum}).
On the other hand we can define a joint spectrum for a finite generating set of $\mathbb{D}(G/K)$ using approximate eigenvectors (see Section~\ref{ssub:Schmüdgen's joint spectrum}). 
This definition is more in the spirit of usual spectral theory.
In fact both definitions coincide and it holds:
\begin{equation}
	\label{eq:defspectrum}
	\widetilde \sigma(\Gamma\backslash G/K)=\{\lambda\in \mf a_\C^\ast\mid \chi_\lambda(D)\in \sigma(_\Gamma D) \quad \forall D\in \mathbb{D}(G/K)\}
\end{equation}
where $\chi_\lambda$ are the characters of $\mathbb{D}(G/K)$ parametrized by $\lambda\in \mf a_\C^\ast$ (see Proposition~\ref{prop:equivalencespectra}).

As a first result we prove that
\begin{equation}
	\label{eq:spectrumcontainsunitary}
	i\mf a^\ast\subseteq \widetilde \sigma(\Gamma\backslash G/K)
\end{equation}
if $\Gamma\backslash G/K$ has infinite injectivity radius (see Proposition~\ref{prop:spectrumcontainsimaginary}).

The above mentioned connection between this spectrum and temperedness of $L^2(\Gamma\backslash G)$ is given by the following fact.
\begin{fact}[Proposition~\ref{prop:temperednessfromspectrum}]
	If $\widetilde \sigma(\Gamma\backslash G/K)\subseteq i\mf a^\ast$ then $L^2(\Gamma \backslash G)$ is tempered.
\end{fact}

Until recently, it was completely unknown which conditions on $\Gamma$ (similar to $\delta_\Gamma\leq \|\rho\|$) imply temperedness of $L^2(\Gamma\backslash G)$ even for the example of $G=G_1\times G_2$ with $G_i$ of rank one.
Then Edwards and Oh \cite{EdwardsOh22} showed temperedness for Anosov subgroups if the growth indicator function $\psi_\Gamma$ is bounded by $\rho$ (see Section~\ref{sec:growthindicator} for the definition). 
This statement is in the same spirit as the original theorems by Patterson, Sullivan, and Corlette,
but it only holds for Anosov subgroups for minimal parabolics which are a higher rank analog of convex cocompact subgroups and its proof uses rather different methods including estimates on mixing rates from \cite{EdwOhMixing}.
The main example where they verify the condition $\psi_\Gamma\leq \rho$ is precisely the product situation $G=G_1\times G_2$ with $G_i$ of rank one and $\Gamma$ is an Anosov subgroup.

In this work we present a different proof for the temperedness of $L^2(\Gamma \backslash (G_1\times G_2))$ that is closer to the original proofs in the rank one case and does not use any mixing results.
Moreover, we need not to assume that $\Gamma$ is Anosov.
\begin{theorem*}[Theorem~\ref{prop:spectrumlaplaceonefactor}]
	Let $G_1$ and $G_2$ be of rank one and $\Gamma\leq G_1\times G_2$ discrete and torsion-free.
	Let 
	\[
		\delta_1=\sup_{R>0} \inf\left\{s\in\R \colon \sum_{\g\in\G, \|\mu_+(\g_2)\|\leq R} e^{-s \|\mu_+(\g_1)\|} <\infty\right\}
	\]
	and define $\delta_2$ in the same way.
Then
\[
		\widetilde \sigma(\G \backslash (G_1\times G_2)/(K_1\times K_2))\subseteq\{\lambda\in\mf a_\C^\ast\mid  \|\Re(\lambda_i)\|\leq \max (0,\delta_i-\|\rho_i\|)\}.
	\]
\end{theorem*}
\begin{figure}[ht]
	\centering
	\includegraphics[width=0.8\textwidth,trim=2cm 4cm 8cm 2cm, clip]{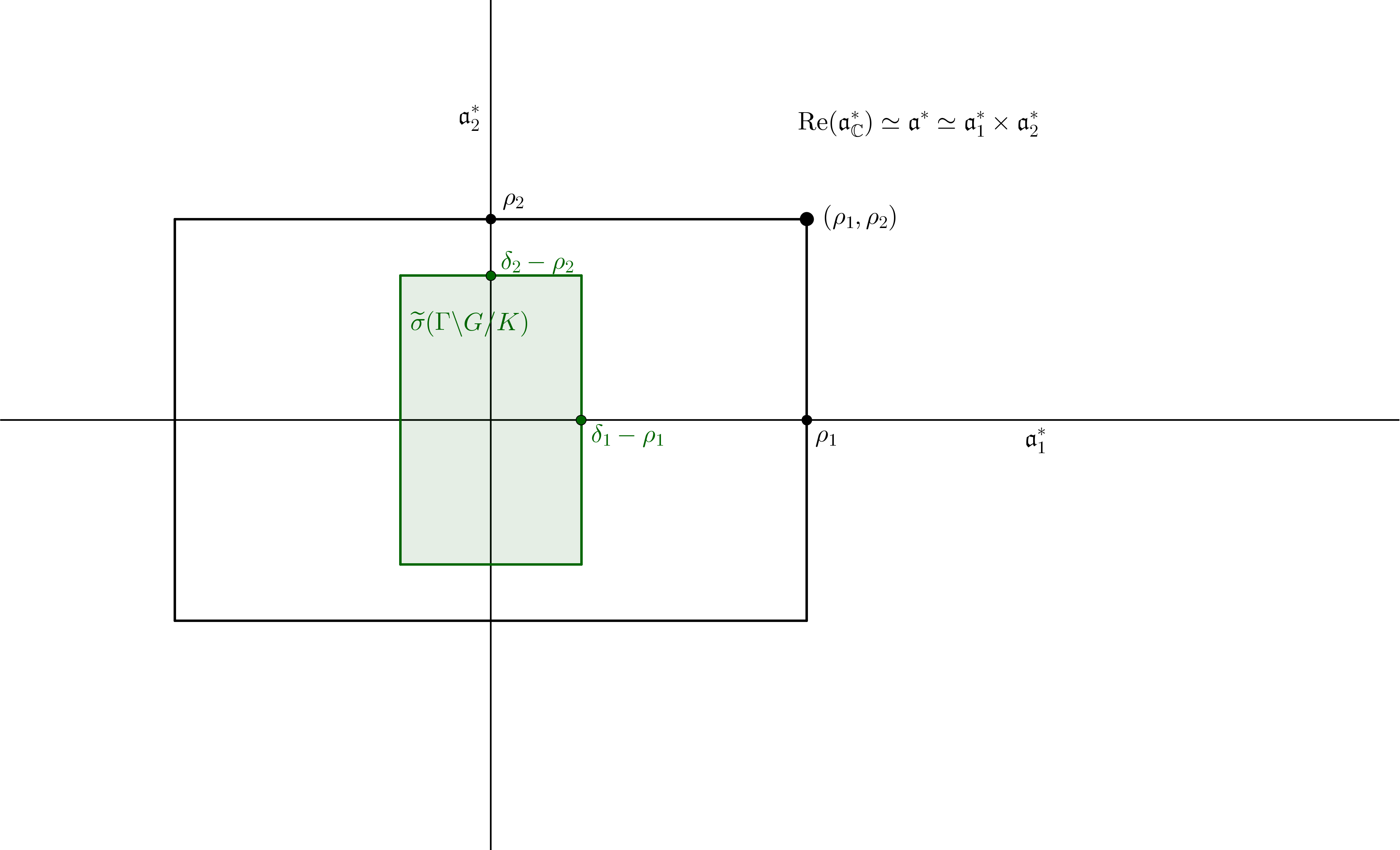}
	\caption{$\widetilde \sigma(\Gamma\backslash (G_1\times G_2)/(K_1\times K_2))$ for two rank one groups $G_i$}
	\label{fig:}
\end{figure}
For the proof we consider the Laplace operators on the two factors and use \eqref{eq:defspectrum} to bound $\widetilde \sigma$.
For these operators the proof is similar to the proofs of Patterson and Corlette, 
i.e. we obtain information about the spectrum by considering the resolvent kernel on the globally symmetric space $G/K$ and get the local version by averaging over $\Gamma$.
Analyzing the region of convergence of this averaging process leads to the theorem.

We obtain the following corollary.
\begin{corollary*}[Corollary~\ref{cor:temperednessProduct}]
	If $\delta_1\leq \|\rho_1\|$ and $\delta_2\leq \|\rho_2\|$ then $L^2(\Gamma\backslash (G_1\times G_2))$ is tempered. 
\end{corollary*}
An important example is a selfjoining: 
Let $\pi_i\colon G_1\times G_2 \to G_i$ be the projection on one factor.
Suppose that $\pi_i|_\Gamma$, $i=1,2$, both have finite kernel and discrete image. 
Then $\delta_1=\delta_2=-\infty$ and 
hence $L^2(\Gamma\backslash (G_1\times G_2))$ is tempered.
Any Anosov subgroup with respect to the minimal parabolic subgroup in $G_1\times G_2$ satisfies this assumption, 
but also satisfies additional assumptions, 
e.g. $\Gamma$ is word hyperbolic and $\|\mu_+(\pi_i(\gamma))\|$ is comparable to the word length of $\gamma\in \Gamma$ \cite{Lab06, GW12}.
Therefore we generalize this part of \cite{EdwardsOh22}.
In contrast, \cite{EdwardsOh22} also provide statements on the connection between temperedness and growth behavior of the Anosov subgroup $\Gamma$ for more general (globally) symmetric spaces $G/K$ which are not products of rank one symmetric spaces.
To extend our work to this more general setting one needs growth estimates for the kernel of the resolvent 
for suitable generators of the algebra $\mathbb{D}(G/K)$ which so far only seem to be known for the Laplace operator (see \cite{ankerji}).

\subsection*{Outline of the article}
In Section~\ref{sec:prelim} we recall some preliminaries about the symmetric space,
spherical functions, the spherical dual, and the Fourier-Helgason transform.
After that we define the Plancherel spectrum (see Section~\ref{sec:planchrelspectrum}) and the joint spectrum (see Section~\ref{ssub:Schmüdgen's joint spectrum}) and show that they coincide (see Proposition~\ref{prop:equivalencespectra}).
We also prove \eqref{eq:spectrumcontainsunitary} in Proposition~\ref{prop:spectrumcontainsimaginary}.
In Section~\ref{sec:temperedness} we show the connection between $\widetilde\sigma(\Gamma\backslash G/K)$ and the temperedness of $L^2(\Gamma\backslash G)$.
We suppose that the statements might be considered as folklore among experts in spectral theory of higher rank symmetric spaces, but as the literature on spectral theory of locally symmetric spaces of higher rank and infinite volume is very sparse we provide precise statements with complete proofs in this section.
In Section~\ref{sec:spectrumonproduct} we prove Corollary~\ref{cor:temperednessProduct}.
To do so we first recall the averaging procedure (see Lemma~\ref{la:resolventkernellocally}) and 
reprove the rank one result by \cite{Cor90} in a form that we need later (see Lemma~\ref{la:nongroup}).
We conclude this article by comparing the quantities $\delta_i$ with the growth indicator function $\psi_\Gamma$ (see Section~\ref{sec:growthindicator}). 
	\section{Preliminaries}
	\label{sec:prelim}
	\subsection{Setting}
	\label{sec:setting}
	In this section we introduce the notation in the general higher rank setting and only restrict to product spaces once it becomes necessary in order to emphasize clearly what the missing knowledge for the general higher rank setting is.
	Let $G$ be a real connected semisimple non-compact Lie group with finite center and  with Iwasawa decomposition $G=KAN$. 
	We denote by $\mf g, \mf a, \mf n, %
	\mf k%
	$ the corresponding Lie algebras. 
	For $g\in G$ let $H(g)$ be the logarithm of the $A$-component in the Iwasawa decomposition $KAN$.   
	We have a $K$-invariant inner product on $\mf g$ that is induced by the Killing form and the Cartan involution. 
	We further have the orthogonal Bruhat decomposition $\mf g =\mf a \oplus \mf m \oplus \bigoplus _{\alpha\in\Sigma} \mf g_\alpha$ into root spaces $\mf g_\alpha$ with respect to the $\mf a$-action via the adjoint action $\ad$. 
	Here $\Sigma\subseteq \mf a^\ast$ is the set of restricted roots. 
	Denote by $W$ the Weyl group of the root system of restricted roots. 
	Let $n$ be the real rank of $G$ and $\Pi$ (resp. $\Sigma^+$) the simple (resp. positive) system in $\Sigma$ determined by the choice of the Iwasawa decomposition.  
	Let $m_\alpha \coloneqq \dim_\R \mf g_\alpha$ and $\rho \coloneqq \frac 12 \Sigma_{\alpha\in \Sigma^+} m_\alpha \alpha$. 
	Let $\mf a_+ \coloneqq \{H\in \mf a\mid \alpha(H)>0 \,\forall \alpha\in\Pi\}$ denote the positive Weyl chamber  and $\mf a^\ast_+$ the corresponding cone in $\mf a^\ast$ via the identification $\mf a \leftrightarrow \mf a^\ast$ through the Killing form $\langle\cdot,\cdot\rangle$ restricted to $\mf a$. 
	If $\ov {A^+} \coloneqq \exp (\ov {\mf a_+})$, then we have the Cartan decomposition $G=K\ov {A ^+}K$. 
	For $g\in G$ we define $\mu_+(g)\in \ov{\mf a_+}$ by $g\in K\exp (\mu_+(g))K$. 
	The main object of our study is the symmetric space $X=G/K$ of non-compact type.

	Let $\mathbb D(G/K)$ be the algebra of \emph{$G$-invariant differential
	operators} on $G/K$, i.e. differential operators commuting with the
	left translation by elements  $g\in G$. Then we have an algebra
	isomorphism $\HC\colon \mathbb D(G/K)\to \text{Poly}(\mf a^\ast)^W$
	from $\mathbb D(G/K)$ to the $W$-invariant complex polynomials on $\mf
	a^\ast$ which is called the \emph{Harish-Chandra homomorphism} (see
	\cite[Ch.~II~Thm. 5.18]{gaga}). For $\lambda\in \mf a^\ast_\C$ let
	$\chi_\lambda$ be the character of $\mathbb D(G/K)$ defined by
	$\chi_\lambda(D)\coloneqq \HC(D)(\lambda)$. Obviously, $\chi_\lambda=
	\chi_{w\lambda}$ for $w\in W$. Furthermore, the $\chi_\lambda$ exhaust
	all characters of $\mathbb D(G/K)$ (see \cite[Ch.~III
	Lemma~3.11]{gaga}). 
	We define the space of joint eigenfunctions
	$$E_\lambda \coloneqq\{f\in C^\infty(G/K)\mid Df = \chi_\lambda (D) f
	\quad \forall D\in  \mathbb D(G/K)\}.$$ Note that $E_\lambda$ is
	$G$-invariant. 

	For example the (positive) Laplace operator $\Delta$ is contained in $\mathbb D(G/K)$ and $\chi_\lambda(\Delta)=-\langle\lambda,\lambda\rangle +\langle \rho,\rho\rangle$.

	\subsection{Spherical functions}
	One can show that in each joint eigenspace $E_\lambda$ there is a unique left $K$-invariant function which has the value $1$ at the identity (see \cite[Ch.~IV~Corollary 2.3]{gaga}). We denote the corresponding bi-$K$-invariant function on $G$ by $\phi_\lambda$ and call it \emph{elementary spherical function}. Therefore, $\phi_\lambda = \phi_\mu$ iff $\lambda =w\mu$ for some $w\in W$. It is given by  $\phi_\lambda (g) = \int _K e^{-(\lambda+\rho)H(g\inv k)} dk$. 
	Note that we differ from the notation in \cite{gaga} by a factor of $i$: $\phi^{\text{Hel}}_\lambda = \phi_{i\lambda}$. 

	\subsection{Functions of positive type and unitary representations}\label{sec:sphericalpos}
	In this section we recall the correspondence between elementary spherical functions of positive type and irreducible unitary spherical representations.
	Recall first that a continuous function $f\colon G\to \C$ is called \emph{of positive type} if the matrix $(f(x_i\inv x_j))_{i,j}$ for all $x_1,\ldots, x_k\in G$ is positive semidefinite. If $f$ is of positive type, then one has  $f(x\inv)=\ov{ f(x)}$ and $|f(g)|\leq f(1)$. Moreover, we can define a unitary representation $\pi_f$ associated to $f$ in the following way: If $R$ denotes the right regular representation of $G$, then $\pi_f$ is the completion of the space spanned by $R(x)f$ with respect to the inner product defined by $\langle R(x)f, R(y)f\rangle\coloneqq f(y\inv x)$ which is positive definite. $G$ acts unitarily on this space by the right regular representation. If $f(g)=\langle\pi(g) v,v\rangle$ is a matrix coefficient of a unitary representation $\pi$, then $f$ is of positive type and $\pi_f$ is contained in $\pi$.

	Secondly, recall that a unitary representation is called \emph{spherical} if it contains a non-zero $K$-invariant vector. Denote by $\widehat G_{\text{sph}}$ the subset of the unitary dual consisting of spherical representations.   We then have a 1:1-correspondence  between elementary spherical functions of positive type and  $\widehat G_{\text{sph}}$ given by $\phi_\lambda\mapsto \pi_{\phi_\lambda}$  (see \cite[Ch.~IV~Thm.~3.7]{gaga}). The preimage of an irreducible unitary spherical representation $\pi$ with normalized $K$-invariant vector $v_K$ is given by $g\mapsto \langle \pi(g)v_K,v_K\rangle$.

	\subsection{Harish-Chandra's c-Function}
	\begin{definition}
		We define the Harish-Chandra $\mathbf c$-function for $\lambda\in \mf a_\C^\ast$ with $\Re \lambda\in \mf a_+^\ast$ as the absolutely convergent integral 
		$$\mathbf c(\lambda) = \int_{\ov N} e^{-(\lambda+\rho)H(\ov n)} d\ov n,$$
		where $d\ov n$ is normalized such that $\cf (\rho)=1$.
		It is given by the product formula 
		\begin{align*}
			\cf(\lambda)=c_0 \prod_{\alpha\in\Sigma_0^+} \frac{2^{-\langle\lambda,\alpha_0\rangle}\G(\langle \lambda, \alpha_0\rangle)}{\G(\frac 14m_\alpha + \frac 12 +\frac 12 \langle \lambda,\alpha_0\rangle)\G(\frac 14m_\alpha + \frac 12m_{2\alpha} +\frac 12 \langle \lambda,\alpha_0\rangle)}
		\end{align*}
		where $\Sigma_0^+ = \Sigma^+\setminus \frac 12 \Sigma^+$, $\alpha_0=\alpha/\langle \alpha,\alpha\rangle$, and the constant $c_0$ is determined by $\cf(\rho)=1$. 
	\end{definition}

	\subsection{The Fourier-Helgason transform}
	For a sufficiently nice function $f\colon G/K\to \C$ we define the Fourier-Helgason transform of $f$ by 
	\begin{align*}
		\mc Ff(\lambda,kM)= \int_{G/K} f(gK)e^{(\lambda-\rho)H(g\inv k)}d(gK).
	\end{align*}

	Let $e_{\lambda,kM}(gK)=e^{-(\lambda+\rho)H(g\inv k)}$. Then we have $D e_{\lambda,kM}=\chi_\lambda(D) e_{\lambda,kM}$ by \cite[Ch.~II Lemma~5.15]{gaga} for every $D\in\mathbb D(G/K)$, $\lambda\in \mf a_\C^\ast$, $k\in K$, and $g\in G$.
	Therefore, \begin{align*}\mc F(Df)(\lambda,kM)&=\int_{G/K}
		Df(gK)\ov{e_{-\ov\lambda,kM}(gK)}d(gK)=\int_{G/K} f(gK)\ov{D^\ast
		e_{-\ov\lambda,kM}(gK)}d(gK)\\
&=\int_{G/K}^{} {f(gK) \ov{\chi_{-\ov \lambda}(D^\ast)e_{-\ov\lambda,kM}(gK)}} \: d(gK)=
\ov{\chi_{-\ov\lambda}(D^\ast)} \mc Ff(\lambda,kM).
	\end{align*} 
	By
	\cite[Lemma~5.21 and Cor.~5.3]{gaga} $\chi_{\lambda}(D^\ast)=
	\ov{\chi_{-\ov\lambda}(D)}$ so that we have the following lemma.
	\begin{lemma}\label{la:FourierIntertwiner}
		The Fourier-Helgason transform satisfies $\mc F(Df)(\lambda,kM)=\chi_{\lambda}(D) \mc Ff(\lambda,kM)$ for every $D\in \mathbb D(G/K)$.
	\end{lemma}

	\begin{theorem}[{\cite[Ch. III Thm.~1.5]{HelGeomAna}}]\label{thm:FourierL2}
		The Fourier-Helgason transform is an isometry between $L^2(G/K)$ and $L^2 (i\mf a_+^\ast \times K/M, |\cf(\lambda)|^{-2} d\lambda d(kM))$. Moreover,
		$$\langle f,g\rangle_{L^2(G/K)} = |W|\inv \int_{i\mf a^\ast\times K/M} \mc Ff(\lambda,kM) \ov {\mc Fg(\lambda,kM)} |\cf(\lambda)|^{-2} d\lambda d(kM)$$
	\end{theorem}

	In particular, Lemma~\ref{la:spectrumdirectintegral} implies that
	$\sigma(D)=\on{essran}[i\mf a^\ast_+\to \C,\lambda\mapsto \chi_\lambda(D)]$ with respect to the measure $|\cf(\lambda)|^{-2}d\lambda d(kM)$.
	Since $\chi_\lambda(D)$ is polynomial and $|\cf(\lambda)|^{-2}>0$ for $\lambda\in i\mf a_+^\ast$ we find that the spectrum of $D$ is the closure of $\{\chi_\lambda(D)\mid \lambda\in i\mf a_+^\ast\}$.
	As $\chi_\lambda(D)$ is $W$-invariant this coincides with the closure of $\{\chi_\lambda(D)\mid \lambda\in i\mf a^\ast\}$.

	\section{Spectra for locally symmetric spaces}
	\label{sec:spectrumlocally}
	In this section we recall different types of spectra for the algebra $\mathbb D (G/K)$ on a locally symmetric space.

	Let $\G\leq G$ be a torsion-free discrete subgroup.
	\subsection{Plancherel spectrum}
	\label{sec:planchrelspectrum}
	We want establish a spectrum for the algebra $\mathbb{D}(G/K)$ of $G$-invariant differential operators.
	Let us start with the spectrum that is obtained from decomposing the representation $L^2(\G\backslash G)$.

\begin{theorem}[{see e.g. \cite[Thm. F.5.3]{propT}}]
	Let $\pi$ be a unitary representation of $G$. Then there exists a standard Borel space $Z$, a propability measure $\mu$ on $Z$, and a measurable field of irreducible unitary representations $(\pi_z,\mc H_z)$ such that $\pi$ is unitarily equivalent to the direct integral $\int^ \oplus_Z \pi_zd\mu(z)$.
\end{theorem}

According to the previous theorem let $L^2(\G\backslash G)$ be the direct integral $\int^\oplus_Z \pi_z d\mu(z)$.
We denote by $Z_{\on{sph}}$ the subset $\{z\in Z\mid \pi_z\text{ is spherical}\}$ of $Z$ where \emph{spherical} means that the representation has a non-zero $K$-invariant vector.
We note that projection $P\colon L^2 (\G\backslash G)\to L^2(\G\backslash G)^K$ onto the $K$-invariant vectors is given by $\int_K R(k)dk$ where $R$ is the representation of $G$ on $L^2(\G\backslash G)$.
Hence, there is a measurable vector field $z\mapsto v_z^K$ such that $v_z^K\in \mc H_z^K$ is of norm $1$ if $\mc H_z^K\neq 0$. 
In particular, $Z_{\on{sph}}$ is measurable.
For $z\in Z_{\on{sph}}$ the representation $\pi_z$ is unitary, irreducible, and spherical. 
	By Section~\ref{sec:sphericalpos} $\pi_z\simeq \pi_{\phi_{\lambda_z }}$ 
	for some $\lambda_z\in \mf a_\C^\ast$ 
	such that $\phi_{\lambda_z}$
	is of positive type.

Recall the definition of the essential range for a measurable function $f\colon (Z,\mu)\to Y$ from a probability space  into a second countable topological space $Y$:
\[
	\on{essran} f\coloneqq \{y\in Y\mid \,\forall \,U\subseteq Y \text{ open}, y\in U \colon \mu(f\inv (U))>0\}.
\]
By definition $\on{essran} f$ equals the support of the pushforward measure $f_\ast\mu$
and for $A\subseteq Y$ closed $\on{essran} f \subseteq A$ if and only if $f(z)\in A$ for $\mu$-a.e. $z\in Z$ which we can see as follows:
Clearly, if $\mu(\{f(z)\notin A\}) = 0$, then 
$\on{essran} f\cap Y\setminus A=\emptyset$.
Hence, $\on{essran} f\subseteq A$.
Conversely, if $\on{essran} f \cap Y\setminus A=\emptyset$ then for every $a\in Y\setminus A$ we find an open neighborhood $N_a$ of $a$ with $\mu(f\inv(N_a))=0$.
Since $Y$ is second countable $Y\setminus A$ can be covered by countably many $N_a$.
Thus $\mu(f\inv(Y\setminus A)) \leq \sum \mu(f\inv (N_a))=0$.
Therefore, $f(z)\in A$ for $\mu$-a.e. $z\in Z$.

The following lemma motivates the definition of the Plancherel spectrum.
\begin{lemma}\label{la:spectrumdirectintegral}
	Let $\mc H  =\int^\oplus_Z \mc H_zd\mu(z)$ be the direct integral of
the field $(\mc H_z)_{z\in Z}$ of Hilbert spaces over the  $\sigma$-finite
measure space $(Z,\mu)$. Let $T=\int^\oplus_Z T_z d\mu(z)$ be the direct
integral of the field of operators $(T_z)_{z\in Z}$ such that $T(z)=f(z)id_{\mc
H_z}$ for a measurable function $f$ where the domain of $T$ is $\{\int_{Z}^{\oplus} {y_z} \: d\mu(z)\in \mc H\mid \int_{Z}^{\oplus} {|f(z)|^2 \|y_z\|^2} \: d\mu(z)<\infty\}$.
Then $$\sigma(T)=\on{ess ran} f = \{y\in
\C\mid \,\forall\, \eps>0\colon \mu (f\inv(B_\varepsilon(y)))>0\}.$$
\end{lemma}
\begin{proof}
	If $\lambda\notin \on{essran} f$ then there is $\varepsilon >0$ such that $|f(z)-\lambda|\geq \eps$ for a.e. $z\in Z$.
	Hence, $\int_Z \frac 1{f(z)-\lambda} id_{\mc H_z}d\mu(z)$ is bounded operator with operator norm $\leq 1/\eps$ inverting $T-\lambda$. 
	Therefore, $\lambda\notin \sigma(T)$.

	Conversely, let $\lambda\in \on{essran}f$ and $\eps>0$.
	Then $A_\eps\coloneqq \{z\in Z\mid|f(z)-\lambda|<\eps\}$ has positive measure and there is a unit vector $y_\eps=\int^\oplus_Z y_{\eps,z} d\mu(z)\in\mc H$ such that $ y_{\eps,z} =  0 $ for $z\notin A_\eps$.
	It follows that $$\|(T-\lambda)y_\eps\|^2=\left\|\int^ \oplus_Z (f(z)-\lambda)y_{\eps,z}d\mu(z)\right\|^2 = \int_{A_\eps} |f(z)-\lambda|^2 \|y_{\eps,z}\|^2d\mu(z)\leq \eps^2.$$
	Consequently, $T-\lambda$ cannot be invertible.
\end{proof}

For a locally symmetric space $\G\backslash G/K$ we define 
\[
	\widetilde \sigma (\G\backslash G/K) \coloneqq \on{essran} [z\mapsto \lambda_z] \subseteq \mf a_\C^\ast/W.
\]
Note that $\widetilde \sigma(\Gamma\backslash G/K)\subseteq \{\lambda\in \mf a_\C^\ast/W\mid \phi_\lambda \text{ is of positive type}\}$.
In particular, since functions of positive type are bounded $\widetilde \sigma(\Gamma\backslash G/K)\subseteq \on{conv}(W\rho)$ (see \cite[Ch.~IV~Thm.~8.1]{gaga}).
Furthermore, $\phi_\lambda=\phi_{-\ov \lambda}$ so that $\widetilde \sigma(\Gamma\backslash G/K)\subseteq \{\lambda\in \mf a_\C^\ast/W\mid -\ov \lambda\in W\lambda\}$ (see e.g.~\cite[Sec.~4]{qcchigherrank}).

\subsection{The joint spectrum}%
\label{ssub:Schmüdgen's joint spectrum}
In this section we describe a different kind of spectrum for $\mathbb{D}(G/K)$ that takes the action of the operators into account instead of the representation theoretical decomposition (see \cite[Ch.~5.2.2]{Schmudgen}).
\begin{definition}[{see \cite[Prop. 5.27]{Schmudgen}}]
	Let $T_1$ and $T_2$ be (not necessarily bounded) normal operators on a Hilbert space $\mc H$.
	We say that $T_1$ and $T_2$ strongly commute if their spectral measures $E_{T_1}$ and $E_{T_2}$ commute.
\end{definition}
For strongly commuting normal operators we can define the following joint spectrum.
\begin{definition}[{see \cite[Prop. 5.24]{Schmudgen}}]
	\label{def:jointspectrum}
	Let $T=\{T_1,\ldots,T_n\}$ be a family of pairwise strongly commuting operators on a Hilbert space $\mc H$.
	We define 
	$	\sigma_{j}(T) $
	to be the set of all $s\in \C^n$ such that there is a sequence
	$(x_k)_{k\in \N}$ of unit vectors in $\bigcap_{i=1}^n \dom(T_i)\subseteq \mc H$
	satisfying 
	\[
		\lim_{k\to \infty} (T_i-s_i)x_k=0
	\]
	for all $i=1,\ldots,n$.
	We call the sequence $(x_k)$ joint approximate eigenvector.	
\end{definition}
Clearly, every joint approximate eigenvector is an approximate eigenvector for $T_i$.
Hence, $s_i\in \sigma(T_i)$ for $s\in \sigma_j(T_i)$ and
({see \cite[Prop. 5.24(ii)]{Schmudgen}}):
	\[
		\sigma_{j}(T)\subseteq \sigma(T_1)\times \cdots \times \sigma(T_n).
	\]

Let us come back to the invariant differential operators on a locally symmetric space.
By definition $D\in\mathbb{D}(G/K)$ is $G$-invariant and 
therefore it maps $\G$-invariant elements in $C^\infty(G/K)$ into itself.
Since ${}^\G C^\infty(G/K) \simeq C^\infty(\Gamma\backslash G/K)$ we obtain a differential operator ${}_\Gamma D$ on $\Gamma \backslash G/K$.
Using the direct integral decomposition it is easy to see that ${}_\G D$ is a normal operator on $L^2(\G\backslash G/K)$ for $D\in \mathbb D(G/K)$ (with domain $\{f\in L^2(\G\backslash G/K)\mid {}_\G D f\in L^2(\G\backslash G/K)\}$).
Furthermore, the spectral measure is given by
\[
	E_{{}_\G D}(M) \int_{Z_{\on{sph}}}^{\oplus} {f_z} \: d\mu(z) = \int_{\{z\mid \chi_{\lambda_z}(D)\in M\}}^{\oplus} {f_z} \: d\mu(z).
\]
We obtain that ${}_\G D_1$ and ${}_\G D_2$ strongly commute for
$D_1,D_2\in \mathbb D(G/K)$ and hence we can define the joint spectrum
for any finite family $\{{}_\G D_1,\ldots,{}_\G D_n\}$.

\subsection{Comparison of spectra}
In this section we want to see that the Plancherel spectrum and the joint spectrum coincide.
In order to achieve this we need the following lemma.

\begin{lemma}\label{la:polyhomeo}
	Let $p_1,\ldots,p_n\in Poly(\mf a_\C^\ast)^W$ be non-constant complex Weyl group invariant homogeneous  
	polynomials of degree $d_i$ on $\mf a_\C^\ast$ that separate the points on $\mf a_\C^\ast/W$.
	Then $\mf a_\C^\ast/W\to \C^n,\, \lambda \mod W\mapsto
(p_1(\lambda),\ldots,p_n(\lambda))$ is a topological embedding.  \end{lemma}
\begin{proof}
	By definition the mapping $\Phi\colon \lambda \mod W\mapsto
	(p_1(\lambda),\ldots,p_n(\lambda))$ is injective and continuous.
	It remains to show that $\Phi\inv$ is continuous, 
	i.e. for $\lambda_n\in \mf a_\C^\ast$ with $\Phi(\lambda_n)\to \Phi(\lambda_0)$ we have 
	$\lambda_n \mod W \to \lambda_0 \mod W$.
	Since the polynomials $p_i$ are homogeneous it is clear that $\Phi(0)=0$ and 
	$0$ is not contained in $\Phi(\{\lambda\in\mf a_\C^\ast\mid \|\lambda\|=1\}/W)$.
	By compactness
	\[
		\|\Phi(\{\lambda\in\mf a_\C^\ast\mid \|\lambda\|=1\}/W)\|_\infty\geq c>0
	\]
	where we use the maximum norm on $\C^n$.
	Now for $\|\lambda\|\geq 1$:
	\[
		\|\Phi(\lambda \mod W)\|_\infty = \max |p_i(\lambda)| = \max \|\lambda\|^{d_i} |p_i(\lambda/\|\lambda\|)|\geq \|\lambda\|^{\min d_i} \max |p_i(\lambda/\|\lambda\|)|\geq c\|\lambda\|^2.
	\]

	For $\Phi(\lambda_n)\to \Phi(\lambda_0)$ it follows that $\|\lambda_n\|$ is bounded: 
	Indeed if $\limsup \|\lambda_n\|=\infty$ then $\infty = \limsup
	c\|\lambda_n\|^2 \leq \limsup \|\Phi(\lambda_n)\|_\infty \leq \|\Phi(\lambda_0)\|_\infty +1.$ 
	Therefore, $\lambda_n$ is contained in the bounded set $B=\{\lambda\in \mf a_\C^ \ast\mid \|\lambda\|\leq r\}$.
	But now $\Phi|_{B/W}\colon B/W\to \C^n$ is injective and continuous and since $B/W$ is compact it is a topological embedding.
	As $\lambda_n, \lambda_0\in B$ we infer $\lambda_n \mod W\to \lambda_0 \mod W$ and the lemma is proved.
\end{proof}

As before let $L^2(\G\backslash G)=\int_Z^\oplus \pi_z d\mu(z)$.
It is clear that $L^2(\G\backslash G/K)=L^2(\G\backslash G)^K =\int_{Z_{\on{sph}}}^\oplus \mc H_z^K d\mu(z)$.
For $z\in Z_{\on{sph}}$ the representation $\pi_z$ is unitary, irreducible, and spherical. 
By Section~\ref{sec:sphericalpos} $\pi_z\simeq \pi_{\phi_{\lambda_z }}$ 
for some $\lambda_z\in \mf a_\C^\ast/W$ 
such that $\phi_{\lambda_z}$
is of positive type.
This reflects that $\widetilde \sigma(\G\backslash G/K)$ is the set of spectral
parameters $\lambda$ occurring in $L^2(\G\backslash G/K)$.  
By definition of
$\pi_{\phi_{\lambda_z }}$ the differential operator $D\in \mathbb D(G/K)$ acts
by $\chi_{\lambda_z}(D)$ on $\mc H_z^K$.

We now aim to show the following proposition.
\begin{proposition}\label{prop:equivalencespectra}
	Let $D_1,\ldots,D_n$ be a generating set for $\mathbb D(G/K)$ consisting of symmetric operators such that their Harish-Chandra polynomials $\HC(D_i)$ are homogeneous. Then the following sets coincide:
	\begin{enumerate}
		\item $\widetilde \sigma(\G\backslash G/K)$ \label{itm:sigmaschlange}
		\item $\{\lambda\mid \,\forall\, D\in\mathbb D(G/K)\colon \chi_\lambda(D)\in \sigma(_\G D)\}$\label{itm:alloperators}
		\item $\{\lambda\mid \,\forall\, p\in Poly(\mf a_\C^\ast)^W \colon p(\lambda)\in \on{essran} [z\mapsto p(\lambda_z)]\}$\label{itm:allinvariantpolynomials}
		\item $\{\lambda\mid \,\forall\, p\in \C[x_1,\ldots, x_n]\colon p(\chi_\lambda(D_1), \ldots,\chi_\lambda(D_n))\in \sigma(_\G p(D_1,\ldots, D_n))\}$\label{itm:allpolynomials}
		\item $\{\lambda\mid (\chi_\lambda(D_1),\ldots, \chi_\lambda(D_n))\in \sigma_j(_\G D_1,\ldots,{}_\G D_n)\}$\label{itm:joint}
		\item $\{\lambda\mid {}_\G\sum_{i=1}^n (D_i-\chi_\lambda(D_i))^\ast(D_i-\chi_\lambda(D_i)) \text{ is not invertible}\}$\label{itm:box}
	\end{enumerate}
\end{proposition}
\begin{proof}
	Clearly, \ref{itm:alloperators},\ref{itm:allinvariantpolynomials} and \ref{itm:allpolynomials} coincide by the Harish-Chandra isomorphism and Lemma~\ref{la:spectrumdirectintegral} and contain $\widetilde\sigma(\G\backslash G/K)$ by continuity of the polynomials $p\in Poly(\mf a_\C^\ast)^W$.
	Taking $p=\sum_{i=1}^n (x_i-\ov{\chi_\lambda(D_i)})(x_i-\chi_\lambda(D_i))$
	we see that \ref{itm:allpolynomials} is contained in \ref{itm:box}.
	To see that \ref{itm:box} is contained in \ref{itm:joint} we observe
	that an approximate eigenvector for the spectral value 0 for
	${}_\G\sum_{i=1}^n (D_i-\chi_\lambda(D_i))^\ast(D_i-\chi_\lambda(D_i))$ is an
	joint approximate eigenvector for all the ${}_\G D_i$ as 
	\[ 
		\sum_{i=1}^n \|(_\G D_i-\chi_\lambda(D_i))f\|^2= \langle{}_\G\sum_{i=1}^n (D_i-\chi_\lambda(D_i))^\ast(D_i-\chi_\lambda(D_i))f,f\rangle .
	\]
	It remains to show that \ref{itm:joint} is contained in $\widetilde \sigma(\G\backslash G/K)$.
	Let $f_n=\int_{Z_{\on{sph}}}^{\oplus} {f_{n,z}} \: d\mu(z)$ be a joint approximate eigenvector for ${}_\G D_1,\ldots,_\G D_n$ and $A_\eps \coloneqq \{z\mid \sum_{i=1}^n |\chi_{\lambda_z}(D_i)-\chi_{\lambda}(D_i)|^2<\eps\}$.
	Then
	\begin{align*}
		0\leftarrow &\sum \|(_\G D_i - \chi_\lambda(D_i))f_n\|^2 = \int_{Z_{\on{sph}}}^{} {\sum_{i=1}^{n} |\chi_{\lambda_z}(D_i)-\chi_\lambda(D_i)|^2 \|f_{n,z}\|^2} \: d\mu(z)\\
			    &\geq \int_{Z_{\on{sph}}\setminus A_\eps}^{} {\eps \|f_{n,z}\|^2} \: d\mu(z) 
	\end{align*}
	but the last expression equals $\eps$ if $\mu(A_\eps)=0$. 
	Hence, $A_\eps$ has positive measure for all $\eps>0$.
	By Lemma~\ref{la:polyhomeo} the preimage of a neighborhood in $\mf a_\C^\ast/W$ of $\lambda$ under $z\mapsto \lambda_z$ contains $A_\eps$ for some $\eps>0$ and therefore has positive measure as well.
	It follows $\lambda\in \widetilde\sigma(\G\backslash G/K)$.
	This completes the proof.
\end{proof}

We now prove that $\widetilde \sigma(\Gamma\backslash G/K)$ contains $i\mf a^\ast$ if the injectivity radius is infinite.
\begin{proposition}
	\label{prop:spectrumcontainsimaginary}
	Suppose that the injectivity radius of $\G\backslash G/K$ is infinite, 
	i.e. for every compact set $C\subseteq G/K$ there is $g\in G$ such that $G/K\to \G\backslash G/K$ restricted to $gC$ is injective.
	Then $i\mf a^\ast\subseteq \widetilde \sigma(\G\backslash G/K)$.
	In particular, $[\|\rho\|^2,\infty[\:\subseteq \sigma(_\G \Delta)$.
\end{proposition}
\begin{proof}
	The proof follows the same idea as \cite[Prop.~8.4]{EdwardsOh22}.
	Let $\lambda\in i \mf a^\ast = \widetilde \sigma(G/K)$. 
	We choose a generating set $D_1,\ldots,D_n$ for $\mathbb{D}(G/K)$ consisting of symmetric operators such that $\HC(D_i)$ are homogeneous. 
	Let $D_{n+1}= (\Delta -\|\rho\|^2)^k$ for $k$ large such that the order of $D_{n+1}$ is bigger than all the orders of $D_1,\ldots,D_n$.
	Denote the elliptic operator $\sum_{i=1}^{n+1} (D_i-\chi_\lambda(D_i))^\ast(D_i-\chi_\lambda(D_i)) $ by $D$.
	By Proposition~\ref{prop:equivalencespectra} there exists $(f_n)_n\subset L^2(G/K)$ with $\|f_n\|_{L^2(G/K)}=1$ and $D f_n \to 0$.
	Since $D$ is elliptic and positive it is essentially self-adjoint on $C_c^\infty(G/K)$.
	In particular, we can assume that $f_n\in C_c^\infty(G/K)$.
	We can now find $g_n\in G$ such that $g_n \on{supp}f_n$ injects into $\Gamma\backslash G/K$.
	Define $\widetilde f_n(\G x)=f_n(g_n\inv x)$ for $x\in g_nK_n$ and $\widetilde f_n(\G x)=0$ else.
	By construction this is well-defined and $\|\widetilde f_n\|_{L^2(\G\backslash G/K)} =\|f_n\|_{L^2(G/K)}$.
	Moreover, $\|	_\G D\widetilde f_n\|_{L^2(\G\backslash G/K)} = \|Df_n\|_{L^2(G/K)}\to 0.$
	This shows $\lambda\in \widetilde \sigma(\Gamma\backslash G/K)$.
	The 'in particular' part follows from Proposition~\ref{prop:equivalencespectra} \ref{itm:alloperators} and $\chi_\lambda(\Delta)=-\langle \lambda,\lambda\rangle + \|\rho\|^2$.
\end{proof}

\begin{remark}
	The assumption in Proposition~\ref{prop:spectrumcontainsimaginary} is satisfied for the following examples:
	\begin{enumerate}
		\item If $G=SL_2(\R)$ and $\Gamma$ is geometrically finite, then infinite injectivity radius is equivalent to infinite volume which is again equivalent to saying that $\Gamma\backslash \mathbb{H}$ has at least one funnel.
		\item If $G$ is simple of real rank at least $2$, then a discrete subgroup $\Gamma\backslash G/K$ has infinite injectivity radius iff $\Gamma$ has infinite covolume by \cite{GelanderInjectivity}. 
		\item If $\Gamma\leq G$ is an Anosov subgroup, then $\Gamma\backslash G/K$ has infinite injectivity radius \cite[Proposition~8.3]{EdwardsOh22}. 
	\end{enumerate}
\end{remark}

\subsection{Temperedness of \texorpdfstring{$L^2(\G\backslash G)$}{L2(Gamma \textbackslash G)}}
\label{sec:temperedness}
We want to obtain a connection between the spectrum and temperedness of $L^2(\G\backslash G)$.
Let us recall the definition of a tempered representation.

\begin{definition}[{see e.g. \cite{Haagerup1988}}]
	\label{def:tempered}
	A unitary representation $(\pi,\mc H_\pi)$ is called \emph{tempered} if one of the following equivalent conditions is satisfied:
	\begin{enumerate}
		\item $\pi$ is weakly contained in $L^ 2(G)$, i.e. 
			any diagonal matrix coefficients of $\pi$ can be approximated, uniformly on compact sets, by convex combinations of diagonal matrix coefficients of $L^2(G)$. 
		\item for any $\eps>0$ the representation $\pi$ is strongly $L^{2+\varepsilon}$ where 
			$\pi$ is called strongly $L^{p}$ if there is a dense subspace $D$ of $\mc H_\pi$  so that for any vectors $v,w\in D$ the matrix coefficient $g\mapsto \langle\pi(g) v,w\rangle$ lies in $L^ {p}(G)$.
	\end{enumerate}
\end{definition}

To characterize temperedness of $L^2 (\G\backslash G)$ we will use the direct integral decomposition (see Section~\ref{sec:planchrelspectrum}).

We will prove the following statement.
\begin{proposition}
	\label{prop:temperednessfromspectrum}
	Suppose that $\widetilde \sigma(\Gamma\backslash G/K) \subseteq \frac{p-2}p \on{conv}(W\rho)$ for some $p\in [2,\infty[$.
	Then $L^ 2(\G\backslash G)$ is strongly $L^{p+\varepsilon}$.
	In particular, $L^2(\Gamma\backslash G)$ is tempered if $\widetilde\sigma(\Gamma\backslash G/K) \subseteq i\mf a^\ast/W$.
\end{proposition}
\begin{proof}
	Let $\eps>0$ and $f_1,f_2\in C_c(\G\backslash G)$ non-negative.
	We have to show that $\int_G |\langle R(g) f_1,f_2\rangle|^{p+\eps}dg$ is finite.
	Obviously, $\langle R(g)f_1,f_2\rangle = \int_{\G\backslash G} f_1(\G hg)f_2(\G h)d \G h$ is bounded by
	$\langle R(g) F_1,F_2\rangle$ where $F_i(\G h) = \max_{k\in K} |f_i(\G hk)|$.
	Hence, it is sufficent to show $\int_G |\langle R(g) f_1,f_2\rangle|^{2+\eps}dg < \infty$ for $K$-invariant $f_1,f_2$.
	We decompose $f_i$ in the direct integral decomposition as $f_i = \int_Z^\oplus f_{i,z} d\mu(z)$. 
	Since we assumed $f_i$ to be $K$-invariant
	we know that $f_{i,z} \in \mc H_z^K$ for $\mu$-a.e. $z\in Z$.
	It follows that we have to integrate only over $Z_{\on{sph}}$.

	For $z\in Z_{\on{sph}}$ the representation $\pi_z$ is unitary, irreducible, and spherical. 
	By Section~\ref{sec:sphericalpos} $\pi_z\simeq \pi_{\phi_{\lambda_z }}$ 
	for some $\lambda_z\in \mf a_\C^\ast$ 
	such that $\phi_{\lambda_z}$
	is of positive type.
	We also have $\langle \pi_z(g)f_{1,z},f_{2,z}\rangle = \phi_{\lambda_z}(g) \cdot \langle f_{1,z},f_{2,z}\rangle $.
	By assumption, $\lambda_z\in \frac{p-2}{p}\on{conv}(W\rho)$ for a.e. $z\in Z_{\on{sph}}$.
	This implies that $\phi_{\lambda_z} \in L^{p+\varepsilon}(G)$ by
	\cite[Prop.~2.4]{qcchigherrank} and even $\int_{G}^{}
	{|\phi_{\lambda_z}|^{p+\varepsilon}} \: dg\leq C_{\varepsilon,p}$ for
	$\mu$-a.e. $z\in Z_{\on{sph}}$ with $C_{\varepsilon,p}$ independent of $z$.

	Now we estimate
	\begin{align*}
		\int_G |\langle R(g) f_1,f_2\rangle|^{p+\eps}dg&\leq 
		\int_G \left(\int_{Z_{\on{sph}}} |\langle \pi_z(g) f_{1,z},f_{2,z}\rangle|d\mu(z)\right)^{p+\eps}dg \\
		&= \int_G \left(\int_{Z_{\on{sph}}} | \phi_{\lambda_z}(g) \langle f_{1,z},f_{2,z}\rangle|d\mu(z)\right)^{p+\eps}dg.  
	\end{align*}
	Using Hölder's inequality we find that 
	\begin{align*}
		\int_{Z_{\on{sph}}} |& \phi_{\lambda_z}(g) \langle f_{1,z},f_{2,z}\rangle|d\mu(z) 
		= \int_{Z_{\on{sph}}} | \phi_{\lambda_z}(g)| |\langle f_{1,z},f_{2,z}\rangle|^{\frac{1}{p+\eps}} |\langle f_{1,z},f_{2,z}\rangle|^{1/q} d\mu(z) \\
				     &\leq \left( \int_{Z_{\on{sph}}} |\phi_{\lambda_z}(g)|^{p+\varepsilon}|\langle f_{1,z},f_{2,z}\rangle|d\mu(z) \right)^{\frac{1}{p+\varepsilon}} \cdot \left(\int_{Z_{\on{sph}}} | \langle f_{1,z},f_{2,z}\rangle| d\mu(z)\right)^{1/q}.
	\end{align*}
where  $\frac{1}{p+\varepsilon}+\frac{1}{q}=1$.

	Therefore,
	\begin{align*}
		\int_G |\langle R(g) f_1,f_2\rangle|^{p+\eps}dg\leq \int_G \int_{Z_{\on{sph}}} |\phi_{\lambda_z}(g)|^{p+\eps} |\langle f_{1,z},f_{2,z}\rangle|d\mu(z) \cdot \left ( \int_{Z_{\on{sph}}}  |\langle f_{1,z},f_{2,z}\rangle|d\mu(z)\right)^{\frac{p+\varepsilon}{q}} dg.
	\end{align*}
	  Using $\int_{G}^{}
	{|\phi_{\lambda_z}|^{p+\varepsilon}} \: dg\leq C_{\varepsilon,p}$ it follows
	\begin{align*}
		\int_G &|\langle R(g) f_1,f_2\rangle|^{p+\eps}dg\leq C_{\eps,p} \int_{Z_{\on{sph}}}  |\langle f_{1,z},f_{2,z}\rangle|d\mu(z) \cdot \left ( \int_{Z_{\on{sph}}}  |\langle f_{1,z},f_{2,z}\rangle|d\mu(z)\right)^{\frac{p+\varepsilon}{q}} \\
		       &\leq C_{\eps,p} \left ( \int_{Z_{\on{sph}}}  |\langle f_{1,z},f_{2,z}\rangle|d\mu(z)\right)^{p+\eps} 
		       \leq C_{\eps,p} \left ( \int_{Z_{\on{sph}}}  \| f_{1,z}\|^ 2 d\mu(z) \int_{Z_{\on{sph}}} \|f_{2,z}\|^2 d\mu(z)\right)^{p+\eps/2} \\
		       &\leq C_{\eps,p} \|f_1\|_{L^2(\G\backslash G)}^{p+\eps}\|f_2\|_{L^2(\G\backslash G)}^{p+\eps}<\infty.
	\end{align*}
	This completes the proof.
\end{proof}

\section{The spectrum for quotients of products of rank one space}
\label{sec:spectrumonproduct}
	\subsection{The resolvent kernel on a locally symmetric space}
	In this subsection we determine the Schwartz kernel of the resolvent on a locally symmetric space in terms of its Schwartz kernel on the global space $G/K$.
	To do this we need the following well-known lemma.
	\begin{lemma}
		The averaging map $\alpha\colon C_c^\infty(G/K)\to C_c^\infty(\G\backslash G/K)$ defined by
		\[
			\alpha f(\G x) = \sum_{\g\in\G} f(\g  x), \quad x\in G/K,
		\]
		is surjective.
	\end{lemma}
	Let us recall that for $D\in \mathbb{D}(G/K)$ we defined the differential operator ${}_\Gamma D$ acting on $L^2(\Gamma\backslash G/K)$.
	The following lemma tells us how the Schwartz kernel of ${}_\Gamma D\inv$ can be expressed provided $D$ is invertible.
	\begin{lemma}\label{la:resolventkernellocally}
		Let $D\in \mathbb D(G/K)$ and suppose that $D$ is invertible as an unbounded operator $L^2(G/K)\to L^2(G/K)$.
		Let $K_{D\inv}\in \mc D'(G/K\times G/K)$ be the Schwartz kernel of $D\inv$.
		Suppose further that ${}_\G D\colon L^2(\Gamma\backslash G/K)\to L^2(\Gamma\backslash G/K)$ is invertible.
		Then the Schwartz kernel $K_{{}_\G D\inv}\in \mc D'(\G\backslash G/K \times
		\G\backslash G/K)$ of ${}_\G D\inv$ is given by 
		\[K_{{}_\G D\inv}(\varphi\otimes \psi) = \sum_{\g\in\G} K_{D\inv}(L_\g \tilde \varphi\otimes \tilde \psi),
		\]
		where $\tilde \varphi$ (and $\tilde \psi)$ are preimages of $\varphi$ (resp. $\psi$) under the surjective map $\alpha\colon C_c^\infty(G/K)\to C_c^\infty(\G\backslash G/K)$.
		By slight abuse of notation we write
		$$K_{{}_\G D\inv}(\G x,\G y) =
		\sum_{\g\in\G} K_{D\inv}(x,\gamma y).$$ 
	\end{lemma}
	\begin{proof}
		First of all note that $D$ and therefore $D\inv$ is $G$-invariant, hence 
		$K_D(\tilde \varphi\otimes \tilde \psi) = K_D(L_g \tilde \varphi \otimes L_g \tilde \psi)$ for all $g\in G$ and $\tilde \varphi,\tilde \psi\in C_c^\infty(G/K)$.
		Let $\varphi = \alpha \tilde \varphi, \psi=\alpha\tilde \psi\in C_c^\infty(\G\backslash G/K)$.
		By definition of $K_{{}_\G D\inv}$ we have
		$K_{{}_\G D\inv} ((_\G D \varphi )\otimes  \psi)= \int_{\G\backslash G/K}^{} {\varphi(\Gamma x)\psi(\G x)} \: d\G x.$
		On the other hand 
		${}_\G D \varphi =  \alpha (D\tilde \varphi)$ by $G$-invariance of $D$ 
		so that we can choose $D\tilde \varphi$ as $\widetilde{ {}_\G D \phi}$.
		Therefore we have to show
		\begin{align*}
			\sum_{\g\in\G} K_{D\inv}(L_\g D\tilde \varphi\otimes \tilde \psi) =
			\int_{\G\backslash G/K}^{} {\varphi(\Gamma x) \psi(\G x)} \: d\G x.	
		\end{align*}
		The left hand side equals
		\[
			\sum_{\g\in \G}^{} K_{D\inv}(DL_\g \tilde\varphi \otimes \tilde \psi)
			=\sum_{\g\in \G}^{} \int_{G/K}^{} {L_\g \tilde \varphi(x)\tilde\psi(x)} \: dx
		\]
		again by $G$-invariance of $D$ and the definition of $K_{D\inv}$.
		Now we can use the definition of the measure of $\G\backslash G/K$ to conclude
		\begin{align*}
			\sum_{\g\in \G}^{} \int_{G/K}^{} {L_\g \tilde \varphi(x)\tilde\psi(x)} \: dx
			= 
			\sum_{\g\in \G}^{} \int_{\G\backslash G/K}^{} {\sum_{\g'\in \G}  \tilde \varphi(\g x)\tilde\psi(\g'x)} \: d\G x
			= \int_{\G\backslash G/K}^{} {\varphi(\G x)\psi(\G x)} \: d\G x.	
		\end{align*}
		This shows the lemma.
\end{proof}

\subsection{Spectrum of the Laplacian in a general locally symmetric space of rank one}
In this section we recall the connection between the bottom of the Laplace spectrum on the locally symmetric space $\Gamma\backslash G/K$  of rank one and the critical exponent of $\Gamma$ which is due to Elstrodt \cite{MR360472, MR360473, MR360474} and Patterson \cite{pattersonlimitset} for $G=SL_2(\R)$, Sullivan \cite{Sul87} for $G=SO_0(n,1)$, and Corlette \cite{Cor90} for general $G$ of rank one.
In the higher rank setting this was generalized by Leuzinger \cite{Leu04}, Weber \cite{Web08}, and  Anker and Zhang \cite{ankerzhang}.
\begin{definition}
	We define the abscissa of convergence/critical exponent for $\G$ as \[\delta_\G \coloneqq \inf\left\{s\in \R \colon \sum_{\g\in\G} e^{-s \|\mu_+( \gamma )\|}<\infty \right\}.\]
\end{definition}
Let us recall the theorem for the bottom of the spectrum on a locally symmetric space of rank one and 
its proof as we will use it later in the proof of Theorem~\ref{prop:spectrumlaplaceonefactor}.
\begin{proposition}\label{prop:spectrumrankone}
	Let $G/K$ be a symmetric space of rank one and $\G$ a torsion-free discrete subgroup. 
	Then 
	\begin{align*}
		\sigma(_\G\Delta)\subseteq \begin{cases}
			[\|\rho\|^2,\infty[&\colon\delta_\G<\|\rho\|\\
			[\|\rho\|^2-(\delta_\G-\|\rho\|)^ 2,\infty[&\colon \delta_\G\geq\|\rho\|.
		\end{cases}
	\end{align*}
\end{proposition}

The main ingredient for the proof of Proposition~\ref{prop:spectrumrankone} is the Green function which is the resolvent kernel $K_{(\Delta-z)\inv}$ for the Laplacian $\Delta$.
It is well-known that $K_{(\Delta-z)\inv}$ is smooth function away from the diagonal.
By the $G$-invariance of $\Delta$ we have $K_{(\Delta-z)\inv}(gx,gy)=K_{(\Delta-z)\inv}(x,y)$
and therefore $K_{(\Delta-z)\inv}(x,y)$ only depends on $\mu_+(x\inv y)$.
This allows us to see $K_{(\Delta-z)\inv}$ as a function on $A$ which has the following global bounds:

\begin{theorem}[{\cite[Thm.~4.2.2]{ankerji}}]\label{thm:globalboundsGreen}
	\hspace{0.5cm}
	\begin{enumerate}
		\item For every $z<b<\|\rho\|^2$ there is a constant $C_{z,b}>0$ such that 
			\begin{align*}
				K_{(\Delta-z)\inv}(e^H) \leq C_{z,b} e^{-(\sqrt{\|\rho\|^2 -b} + \|\rho\|)\|H\|}
			\end{align*}
			for all $H\in\mf a$ away from the origin.
		\item For every $z<\|\rho\|^2$ there is a constant $C_z$ such that 
			\begin{align*}
				K_{(\Delta-z)\inv}(e^H) \leq C_{z} 
				\begin{cases}
					\|H\|^{2-\dim(G/K)}&\colon \dim(G/K)>2\\
					\log(1/\|H\|) & \colon \dim(G/K)=2
				\end{cases}
			\end{align*}
			for all $H\in \mf a$ near the origin.
	\end{enumerate}
\end{theorem}

\begin{remark} \label{rmk:trivialboundsGreen}
	In addition to the bounds on $ K_{(\Delta-z)\inv}$ from Theorem~\ref{thm:globalboundsGreen} we will use the following general estimates:
	\begin{align*}
		|K_{(\Delta-z)\inv}|\leq K_{(\Delta-\Re z)\inv}
	\end{align*}
	which is positive.
	Moreover,
	\begin{align*}
		K_{(\Delta-z)\inv}\leq K_{(\Delta-z')\inv}\quad\text{for}\quad z\leq z'< \|\rho\|^2.
	\end{align*}
	These estimates can been seen e.g. by writing $(\Delta-z)\inv$ in terms of the Laplace transform.
\end{remark}

In order to decide whether the kernel given by the averaging construction of Lemma~\ref{la:resolventkernellocally} defines a bounded inverse on $L^2(\G\backslash G/K)$ we use Stone's formula.
\begin{proposition}[{see e.g. \cite[Prop.~5.14]{Schmudgen}}]
	\label{prop:StoneFormula}
	Let $A$ be a self-adjoint operator and $P_I$ the spectral projector of $A$ for a Borel subset $I\subseteq \R$. 
	Then $$\frac 12 (P_{[a,b]}+P_{]a,b[}) = \lim_{\varepsilon\to 0} \frac 1 {2\pi i} \int_a^b (A-(z+i\varepsilon))\inv - (A-(z-i\varepsilon))\inv dz.$$
	Here the limit as $\varepsilon\to 0$ is understood as a strong limit.
\end{proposition}

The advantage of Stone's formula is that the occurring inverted operators are well-defined by the self-adjointness of $A$.   
Hence we can merely consider the Schwartz kernel without having to wonder whether this kernel defines a bounded operator on $L^2$.

\begin{proof}
	[Proof of Prop.~\ref{prop:spectrumrankone}]

	According to Proposition~\ref{prop:StoneFormula}
	we have to determine for which $b<\|\rho\|^2$:
	\begin{align}\label{eq:intstoneLaplace}
		\int_0^b {( {}_\G \Delta-(z+i\eps))\inv} -{( {}_\G \Delta-(z-i\eps))\inv}\:dz \to 0
	\end{align}
	in the strong sense as $\eps\to 0$.
	As in Lemma~\ref{la:resolventkernellocally} denote the Schwartz kernel of $(_\G D -(z\pm i\varepsilon))\inv$ by $K_{({}_\G D -(z\pm i\varepsilon))\inv}$. 
	Then we need to see that 
	\begin{align}\label{eq:intprodstonerankone}
		\int_0^b (K_ {( {}_\G \Delta-(z+i\eps))\inv} -K_ {( {}_\G \Delta-(z-i\eps))\inv})( \varphi \otimes \psi) dz \to 0
	\end{align}
	as $\eps\to 0$ for every $\varphi,\psi\in C_c^\infty(\G\backslash G/K)$ for certain $b<\|\rho\|^2$.
	Let $\tilde \varphi$ (resp. $\tilde \psi$) be a preimage of $\varphi$ (resp. $\psi$) under the map $\alpha$.
	Then the expression in  \eqref{eq:intprodstonerankone} equals 
	\begin{equation}
		\label{eq:expandedstone}
		\int_{0}^{b} {\sum_{\g\in\G}(K_ {  (\Delta-(z+i\eps))\inv} -K_ {  (\Delta-(z-i\eps))\inv})( L_\g \tilde \varphi \otimes \tilde \psi)} \: dz
	\end{equation}	
	by Lemma~\ref{la:resolventkernellocally} since $\Delta$ is symmetric and therefore ${}_\G\Delta-(z\pm i\eps)$ is invertible.

	The following slightly more general lemma shows that \eqref{eq:intstoneLaplace} holds for 
	$b < \|\rho\|^2 - (\max \{0, \delta_\Gamma-\|\rho\|\})^2$
	and
	hence $\sigma(_\G \Delta)\cap (-\infty, \|\rho\|^2-(\max\{0,\delta_\G-\|\rho\|\})^2)=\emptyset$.
\end{proof}
	\begin{lemma}
	\label{la:nongroup}
		Let $D$ be a multiset whose underlying set is a discrete subset of a rank one Lie group $G$ and 
		\[\delta_D \coloneqq \inf\left\{s\in \R \colon \sum_{\g\in D} e^{-s \|\mu_+( \gamma )\|}<\infty \right\}.\]
	For $b < \|\rho\|^2 - (\max \{0, \delta_D-\|\rho\|\})^2$
	it holds that 
\begin{equation*}
		\int_{0}^{b} {\sum_{\g\in D}(K_ {  (\Delta-(z+i\eps))\inv} -K_ {  (\Delta-(z-i\eps))\inv})( L_\g \tilde \varphi \otimes \tilde \psi)} \: dz
		\to 0
	\end{equation*}
	as $\varepsilon\to 0$ for every $\tilde\varphi,\tilde \psi\in C_c^\infty(G/K)$.
	\end{lemma}
	\begin{proof}
	Since the supports of $\tilde \varphi$ and $\tilde \psi$ are compact there are only finitely many $\g\in \G$ such that $\on{supp}(L_\g \tilde \varphi \otimes \psi)$ intersects the diagonal in $G/K\times G/K$ non-trivially.
	For these finitely many $\g\in \Gamma$ the term converges to $0$ as $\Delta - z$ is invertible on $L^2(G/K)$ for $z < \|\rho\|^2$ and therefore $(\Delta-(z\pm i\eps))\inv \to (\Delta -z)\inv$.

	For the other $\g$ we use that $K_{(\Delta - z)\inv}$ is a smooth function away from the diagonal and the estimates from Remark~\ref{rmk:trivialboundsGreen}.
	\begin{align*}
&	\left|		\int_{0}^{b} {\sum_{\g}(K_ {  (\Delta-(z+i\eps))\inv} -K_ {  (\Delta-(z-i\eps))\inv})( L_\g \tilde \varphi \otimes \tilde \psi)} \: dz\right|\\
&\leq \sup_{0\leq z\leq b} b \sum_\g \left |( K_ {  (\Delta-(z+i\eps))\inv} -K_ {  (\Delta-(z-i\eps))\inv})( L_\g \tilde \varphi \otimes \tilde \psi)\right|\\
&\leq \sup_{0\leq z\leq b} b \sum_\g\int_{G/K}\int_{G/K} \left |( K_ {  (\Delta-(z+i\eps))\inv}(x,y) -K_ {  (\Delta-(z-i\eps))\inv}(x,y))( \tilde \varphi(\g\inv x)\tilde \psi(y))\right|\:dx\:dy\\
&\leq \sup_{0\leq z\leq b} 2b \sum_\g\int_{G/K}\int_{G/K} \left | K_ {  (\Delta-z)\inv}(\g x,y)  \tilde \varphi( x)\tilde \psi(y)\right|\:dx\:dy\\
&\leq  2b \sum_\g\int_{G/K}\int_{G/K} \left | K_ {  (\Delta-b)\inv}(\g x,y)  \tilde \varphi( x)\tilde \psi(y)\right|\:dx\:dy\\
	\end{align*}
	Since the Green function only depends on $\mu_+(y\inv \g x)$ this can be estimated by a constant times 
	\[
		\sup_{x,y\in C}\sum_\g |K_{(\Delta-b)\inv} (e^{\mu_+(y\inv \g x)})|
	\]
	where $C\subseteq G$ is compact. 
	Now we use Theorem~\ref{thm:globalboundsGreen} to see that this is bounded for any $\nu>0$  by \begin{equation}
		\label{eq:priortotriangle}
		C_\nu \sup_{x,y\in C} \sum_\g e^{-(\sqrt{\|\rho\|^2-b}+\|\rho\| -\nu)\|\mu_+(y\inv \g x)\|}.
	\end{equation}
	By the triangle inequality 
\[
	\|\mu_+(\gamma)\| \leq \|\mu_+(y)\| + \|\mu_+( x )\|+	 \|\mu_+(x\inv \g y)\|  
\]
so that \eqref{eq:priortotriangle} is bounded by
\[
		C_\nu \sup_{x,y\in C}e^{ (\sqrt{\|\rho\|^2-b}+\|\rho\| -\nu)(\|\mu_+(y)\|+\|\mu_+(x)\|)}\sum_\g e^{-(\sqrt{\|\rho\|^2-b}+\|\rho\| -\nu)\|\mu_+( \g )\|}.
	\]
	This is finite (for small $\nu$) if $\sqrt{\|\rho\|^2 -b} +\|\rho\| > \delta_{\G}$, 
	i.e. $b< \|\rho\|^2-(\max\{0,\delta_D-\|\rho\|\})^2$.

	This estimate allows us to use Lebesgue's dominated convergence theorem to 
	conclude the lemma.
\end{proof}

	Note that in Lemma~\ref{la:nongroup} $D$ is not assumed to be a group.
	We will use this general statement in the proof of 
	Proposition~\ref{prop:spectrumlaplaceonefactor}.

\subsection{Product of rank one spaces}

Let $X=X_1\times X_2 = (G_1\times G_2)/(K_1\times K_2)$ be the product of two rank one symmetric spaces and $\G\subseteq G_1\times G_2$ discrete and torsion-free.
In order to determine $\widetilde \sigma(\Gamma\backslash G/K)$ in this case we bound the spectrum of the Laplacian acting on one factor and then use Proposition~\ref{prop:equivalencespectra}.
\begin{theorem}
	\label{prop:spectrumlaplaceonefactor}
	Let $\Delta_1$ be the Laplacian $\Delta\otimes id$ on $L^2(X_1\times X_2)=L^2(X_1)\otimes L^2(X_2)$ acting  on the first factor.
	Let
	\[
		\delta_1=\sup_{R>0} \inf\left\{s\in\R \colon \sum_{\g\in\G, \|\mu_+(\g_2)\|\leq R} e^{-s \|\mu_+(\g_1)\|} <\infty\right\}.
	\]
	Then
	\[
		\widetilde \sigma(_\G\Delta_1)=\{\lambda\in\mf a_\C^\ast/W\mid \chi_\lambda(\Delta_1)\in\sigma(_\G \Delta_1)\} \subseteq \{\lambda\in\mf a_\C^\ast/W\mid \|\Re(\lambda_1)\|\leq \max (0,\delta_1-\|\rho_1\|)\}.
	\]
\end{theorem}
\begin{proof}
	Since the Schwartz kernel of the identity is the Dirac distribution $\delta_{x_2=y_2}$ on the diagonal in $X_2\times X_2$,
	the Schwartz kernel of $ (\Delta_1-z)\inv$ is
	\[
		K_ { (\Delta_1-z)\inv}((x_1,x_2),(y_1,y_2))= K_ {(\Delta-z)\inv}(x_1,y_1)\delta_{x_2=y_2}(x_2,y_2)
	\]
	for $z\notin [\|\rho_1\|^2,\infty[$.
	Therefore, if $(_\G \Delta_1 -z)$ is invertible the kernel of $
	({}_\G \Delta_1-z)\inv$ is
	\[
		K_ {  ({}_\G\Delta_1-z)\inv}(\G(x_1,x_2),\G(y_1,y_2))=
		\sum _ {\g\in \G} K_ {(\Delta-z)\inv}(\g_1 x_1,y_1)\delta_{x_2=y_2}(\g_2 x_2,y_2)
	\]
	by
	Lemma~\ref{la:resolventkernellocally}. 
	According to Proposition~\ref{prop:StoneFormula}
	we have to determine for which $b<\|\rho_1\|^2$:
	\begin{align*}%
		\int_0^b {( {}_\G \Delta_1-(z+i\eps))\inv} -{( {}_\G \Delta_1-(z-i\eps))\inv}\:dz \to 0
	\end{align*}
	in the strong sense as $\eps\to 0$.
	As in Lemma~\ref{la:resolventkernellocally} denote the Schwartz kernel of $(_\G D -z)\inv$ by $K_{({}_\G D -z)\inv}$. 
	Then we need to see for which $b<\|\rho_1\|^2$
	\begin{align}\label{eq:intprodstone}
		\int_0^b (K_ {( {}_\G \Delta_1-(z+i\eps))\inv} -K_ {( {}_\G \Delta_1-(z-i\eps))\inv})( \varphi \otimes \psi) dz \to 0
	\end{align}
	as $\eps\to 0$ for every $\varphi,\psi\in C_c^\infty(\G\backslash G/K)$.
	Let $\tilde \varphi$ (resp. $\tilde \psi$) be a preimage of $\varphi$ (resp. $\psi$) under the map $\alpha$.
	Then the expression in  \eqref{eq:intprodstone} equals 
	\[
		\int_{0}^{b} {\sum_{\g\in\G}(K_ {  (\Delta_1-(z+i\eps))\inv} -K_ {  (\Delta_1-(z-i\eps))\inv})( L_\g \tilde \varphi \otimes \tilde \psi)} \: dz
	\]	
	by Lemma~\ref{la:resolventkernellocally}.
	Without loss of generality we can assume that $\tilde \varphi = \tilde \varphi_1 \otimes \tilde \varphi_2\in C_c^\infty(X_1)\otimes C_c^\infty(X_2)\subseteq C_c^\infty(X_1\times X_2)$ and in the same way for $\tilde \psi$.
	Then \eqref{eq:intprodstone} reduces to
	\[
		\int_{0}^{b} {\sum_{\g\in\G} \left( K_ {  (\Delta-(z+i\eps))\inv} -K_ {  (\Delta-(z-i\eps))\inv})( L_{\g_1} \tilde \varphi_1\otimes \tilde \psi_1) \right) (\delta _{x_2=y_2}(L_{\g_2}\tilde \varphi_2 \otimes \tilde \psi_2))} \: dz
	\]
	The latter part of the integrand is $\int_{X_2}^{} {\tilde \varphi_2(\g_2\inv x) \tilde \psi_2(x)} \: dx$
	which vanishes if $\g_2$ is large depending on $\tilde\varphi_2$ and $\tilde \psi_2$.
	More precisely, this is the case if
	\[
		\|\mu_+(\gamma_2)\|>2 \max_{x\in \on{supp} \tilde \varphi_2} d(x,eK_2) + \max_{\substack {x\in \on{supp} \tilde\varphi_2\\y\in \on{supp}\tilde \psi_2}}d(x,y) \eqqcolon R.
	\]
	Indeed, $d(x,\gamma_2\inv x)\geq d(\gamma_2 K_2,eK_2)-2 d(x,eK_2)> \max_{\substack {x\in \on{supp} \tilde\varphi_2\\y\in \on{supp}\tilde \psi_2}}d(x,y)$ so that $x\in \on{supp}\tilde \psi_2$ excludes $\gamma_2\inv x\in \on{supp} \tilde \varphi_2$.

	Let $\G_R\coloneqq \{\g\in \G\mid \|\mu_+(\g_2)\|\leq R\}$. 
	It follows that \eqref{eq:intprodstone} is bounded by a constant times
	\[
		\int_{0}^{b} {\sum_{\g\in\G_R}(K_ {  (\Delta-(z+i\eps))\inv} -K_ {  (\Delta-(z-i\eps))\inv})}( L_{\g_1} \tilde \varphi_1\otimes \tilde \psi_1) \: dz
	\]
	Now Lemma~\ref{la:nongroup} yields that this vanishes as $\eps\to 0$ as long as $b < \|\rho_1\|^2 - (\max \{0, \delta_{\on{pr}_1(\Gamma_R)}-\|\rho_1\|\})^2$ where $\on{pr}_1(\Gamma_R)$ is the multiset of $\gamma_1\in G$ with multiplicity $\#\{(\gamma_1',\gamma_2')\in \Gamma_R\mid \gamma_1=\gamma_1'\}$.
	In order to get \eqref{eq:intprodstone} for every $\varphi,\psi$ the above condition on $b$ has to hold for every $R>0$,
	i.e. 
$b < \|\rho_1\|^2 - (\max \{0, \delta_1-\|\rho_1\|\})^2$.
	We infer that $$\sigma(_\G \Delta_1)\subseteq\begin{cases} [\|\rho_1\|^2,\infty[& \colon\delta_1\leq\|\rho_1\|\\
		[\|\rho_1\|^2 - (\delta_1-\|\rho_1\|)^2,\infty[&\colon\delta_1\geq \|\rho_1\|.
	\end{cases}
	$$
	Reformulating this statement in terms of $\widetilde\sigma$ we obtain the stated result.
	\end{proof}

	Obviously, Theorem~\ref{prop:spectrumlaplaceonefactor} is also true if we consider the Laplacian on the second factor with the critical exponent
\[
		\delta_2=\sup_{R>0} \inf\left\{s\in\R \colon \sum_{\g\in\G, \|\mu_+(\g_1)\|\leq R} e^{-s \|\mu_+(\g_2)\|} <\infty\right\}.
	\]
	Using this, Proposition~\ref{prop:equivalencespectra}, and Proposition~\ref{prop:temperednessfromspectrum} we obtain the following corollary giving us temperedness of $L^2(\G\backslash G)$ in dependence of $\delta_1$ and $\delta_2$.
	
	\begin{korollar}
		\label{cor:temperednessProduct}
		If $\delta_1\leq \|\rho_1\|$ and $\delta_2\leq\|\rho_2\|$, then $L^2(\G\backslash G)$ is tempered.
	\end{korollar}

	\begin{example}\label{ex:counting}
		\begin{enumerate}
			\item  Let $\G$ be a product $\G_1\times \G_2$ where each $\G_i\leq G_i$ is discrete and torsion-free. Then it is clear that $\delta_i=\delta_{\G_i}$.
				Hence, we obtain the expected results in this product situation.
			\item Let $\G$ be a selfjoining: both projections $\pi_i\colon G_1\times G_2\to G_i$ onto one factor restricted to $\Gamma$ have finite kernel and discrete image.
				Then the set of $\gamma\in \Gamma$ where $\|\mu_+(\pi_i(\gamma))\|\leq R$ is finite. 
				Therefore $\delta_i=-\infty$ and $L^ 2(\G \backslash G)$ is tempered. 
				
			\item Let $\Gamma\leq G_1\times G_2$ be an Anosov subgroup with respect to the minimal parabolic subgroup, i.e.
				$\Gamma$ is a selfjoining such that $\pi_i|_\Gamma$ are convex-cocompact representations.
				In particular, 
				$L^2(\Gamma\backslash G)$ is tempered.
		\end{enumerate}

	\end{example}

	\subsection{Growth indicator function}
	\label{sec:growthindicator}
	In this section we will take a look at the limit cone and the growth indicator function $\psi_\G$ introduced by Quint \cite{Quint} 
	and compare it with $\delta_1$.
		\begin{definition}
		The limit cone $\mc L_\G$ of $\G$ is defined as the asymptotic cone of $\mu_+(\G)$, i.e.
		\begin{align*}
			\mc L_\G=\{\lim t_n\mu_+(\g_n)\mid t_n\to 0,\g_n\in \G\}.
		\end{align*}
	\end{definition}
	For $\G$ Zariski dense, $\mc L_\Gamma$ is a convex cone with non-empty interior \cite{Benoist1997}.
		From this definition we obtain the following proposition.
\begin{proposition}
	\label{prop:resultlimitcone}
			Let $\Gamma$ be a torsion-free discrete subgroup of $G=G_1\times G_2$ where $G_i$ are of real rank one.
			If $\mc L_\Gamma\subseteq \mf a_+ \cup \{0\}$, then $L^2(\Gamma\backslash G)$ is tempered.
		\end{proposition}
		\begin{proof}
			In view of Corollary~\ref{cor:temperednessProduct} it is sufficient to show that $\delta_i=-\infty$. 
			Suppose there are infinitely many  $\gamma_n\in \Gamma$ pairwise distinct such that $\|\mu_+(\gamma_{n,2})\| \leq R$.
			By discreteness $\|\mu_+(\gamma_{n,1})\|\to \infty$. 
			Hence we can choose $t_n\coloneqq 1/\|\mu_+(\gamma_{n,1})\|$.
			Then $t_n \mu_+(\gamma_{n})$ converges to $(H_1,0)$ where $H_1\in \mf a_{1,+}$ is normalized contradicting $\mc L_\Gamma \subseteq \mf a_+\cup \{0\}$.
			Therefore, there are only finitely many $\gamma\in \Gamma$ with bounded second component and hence $\delta_1=-\infty$.
			The same argument works for $\delta_2$.
		\end{proof}
	For $\G\leq G$ discrete and Zariski dense let $\psi_\G\colon\mf a\to \R\cup\{-\infty\}$ be defined by 
	\begin{align*}
		\psi_\G(H)\coloneqq \|H\| \inf_{H\in\mc C} \inf\{s\in\R\mid \sum_{\g\in \G,\mu_+(\g)\in\mc C} e^{-s\|\mu_+(\g)\|} <\infty\}
	\end{align*}
	where the infimum runs over all open cones $\mc C$ containing $H$ and $\|\cdot\|$ is a Weyl group invariant norm on $\mf a$. 
	For $H=0$ let $\psi_\G(0)=0$. 
	Note that $\psi_\G$ is positive homogeneous of degree $1$.
	In general we have the upper bound $\psi_\G\leq 2\rho$.
By \cite{Quint} we know that $\psi_\G\geq 0$ on $\mc L_\Gamma$, $\psi_\G>0$ on the interior of $\mc L_\Gamma$ and $\psi_\G=-\infty$ outside $\mc L_\Gamma$.
	Moreover, $\psi_\G$ is concave and upper-semicontinuous.

		Let us compare $\delta_1$ to $\psi_\G$ in the situation $G=G_1\times G_2$ where $G_i$ is of real rank one.
		Let $H_i\in \mf a_{i,+}$ of norm $1$ and consider the maximum norm on  $\mf a =\mf a_1\times \mf a_2$.
		In this situation it is clear that $\delta_1\leq \psi_\G(H_1,0)$ since every cone $\mc C$ containing $(H_1,0)$ contains the strip $\mf a_{1,+}\times\{H\in\mf a_{2,+}\mid \|H\|\leq R\}$ outside a large enough compact set.

		Note that if $\psi_\Gamma \leq \rho$ then by the above comparison this condition implies $\delta_i\leq \|\rho_i\|$ which is enough to obtain:
\begin{korollar}
Let $X=X_1\times X_2 = (G_1\times G_2)/(K_1\times K_2)$ be the product of two rank one symmetric spaces and $\G\leq G_1\times G_2$ discrete and torsion-free. If $\psi_\Gamma\leq \rho$ then $L^2(\Gamma\backslash G)$ is tempered.
\end{korollar}
Note that this is precisely the result of \cite{EdwardsOh22} without the assumption that $\Gamma$ is the image of an Anosov representation with respect to a minimal parabolic subgroup.

	\clearpage
	\appendix

	\bibliographystyle{amsalpha}
	\bibliography{literatur}

	\bigskip
	\bigskip

	\end{document}